\documentclass[a4paper]{article}
\usepackage{latexsym}
\usepackage{amsmath,amssymb,amsfonts}
\renewcommand{\baselinestretch}{1.2}
\DeclareMathOperator{\Spec}{Spec}
\DeclareMathOperator{\im}{im}
\DeclareMathOperator{\Sch}{Sch}
\DeclareMathOperator{\Coh}{Coh}
\DeclareMathOperator{\Obj}{Obj}
\DeclareMathOperator{\QCoh}{QCoh}
\DeclareMathOperator{\Mor}{Mor}
\DeclareMathOperator{\I}{I}
\DeclareMathOperator{\II}{II}
\DeclareMathOperator{\Sets}{Sets}
\newtheorem{prethm}{{\bf Theorem}}

\newenvironment{thm}{\begin{prethm}{\hspace{-0.5
               em}{\bf.}}}{\end{prethm}}
\newtheorem{prepro}[prethm]{Proposition}
\newenvironment{pro}{\begin{prepro}{\hspace{-0.5
               em}{\bf.}}}{\end{prepro}}
\newtheorem{prelem}[prethm]{Lemma}
\newenvironment{lem}{\begin{prelem}{\hspace{-0.5
               em}{\bf.}}}{\end{prelem}}

\newtheorem{precor}[prethm]{Corollary}

\newtheorem{preremark}{{\bf Remark}}

\newenvironment{rem}{\begin{preremark}\em{\hspace{-0.5
              em}{\bf.}}}{\end{preremark}}
\newtheorem{preexample}{{\bf Example}}

\newenvironment{example}{\begin{preexample}\em{\hspace{-0.5
               em}{\bf.}}}{\end{preexample}}

\newtheorem{preproof}{{\bf Proof.}}

\newenvironment{proof}[1]{\begin{preproof}{\rm
               #1}\hfill{$\Box$}}{\end{preproof}}

\renewcommand{\thefootnote}

\title{\bf{On the Smoothness of Functors}}
\author{A. Bajravani \\
{\footnotesize {Department of Pure Mathematics, Faculty of Mathematical Sciences,}}\\
{\footnotesize {Tarbiat Modares University, Tehran, Iran }}\\
{\footnotesize {P. O. Box 14115-134}}\\
A. Rastegar$^*$\\
{\footnotesize {Faculty of Mathematics, Sharif University, Tehran, Iran}}\\
{\footnotesize {P. O. Box 11155}}}

\begin{document}
\footnotetext{E-mail Addresses:\\
{\tt bajravani@modares.ac.ir}\\
{\tt rastegar@sharif.edu.ir}\\
{\tt (*)Corresponding Author}}
\date{}
\maketitle
\begin{quote}
{\small \hfill{\rule{13.3cm}{.1mm}\hskip2cm}
\textbf{Abstract}\vspace{1mm}

{\renewcommand{\baselinestretch}{1}
\parskip = 0 mm
 In this paper we will try
 to introduce a good smoothness notion
for a functor.
We consider properties and conditions from geometry and algebraic geometry which we expect a smooth
functor should have.\\

\noindent{\small {\it \bf{Keywords}}:  Abelian Category, First Order Deformations,
Multicategory, Tangent Category, Topologizing
Subcategory.
\\

\noindent{\bf{Mathematics Subject Classification:}}
14A20, 14A15, 14A22.}}
}

\vspace{-3mm}\hfill{\rule{13.3cm}{.1mm}\hskip2cm}
\end{quote}

\section{Introduction}
 Nowadays noncommutative algebraic geometry is in the focus
of many basic topics in mathematics and mathematical physics. In these fields, any under consideration space
is an abelian category and a morphism between
noncommutative spaces is a functor between abelian categories. So one may ask to
generalize some aspects of morphisms between commutative spaces to morphisms
between noncommutative ones. One of the important aspects in commutative case is the notion
of
smoothness of a morphism which is stated in some languages, for example: by
lifting property as a universal language, by projectivity of relative cotangent
sheaves as an algebraic language and by inducing a surjective morphism on tangent
spaces as a geometric language.

In this paper, in order to generalize the notion of smooth morphism to
a functor we propose three different approaches.
A
glance description for the first one is as follows: linear approximations of a
space are important and powerful tools. They have geometric meaning and
algebraic structures such as the vector space of the first order deformations of a space.
 So it is legitimate to consider functors which preserve linear
approximations. On the other hand first order deformations are good candidates
for linear approximations in categorical settings. These observations make it reasonable to
consider functors which preserve first order deformations.\\
The second one is motivated from both Schlessinger's approach
and simultaneous deformations. Briefly speaking, a simultaneous deformation
is a deformation which deforms some ingredients
 of an object simultaneously.
Deformations of morphims with nonconstant target, deformations of a
couple $(X,\mathcal{L})$, in which X is a scheme and $\mathcal{L}$
 is a line bundle on X,
 are examples of such deformations. Also
we see that by
 this approach one can get a
 morphism of moduli spaces of some moduli families.
 We get this,
by fixing a universal ring for objects which correspond to each other by a
smooth
functor. Theorem \ref{Th2} connects this notion to the universal ring of an object. In $3.1$ and $3.2$
we describe geometrical setting and usage of this approach respectively.\\
The third notion of smoothness comes from a basic reconstruction theorem of A. Rosenberg,
influenced by ideas of A. Grothendieck. We think that this approach can be a source to
translate other notions from commutative case to noncommutative one.
In remarks \ref{rem2} and \ref{rem3} we notice that these three smoothness notions
are independent of each other.\\

Throughout this paper $\mathbf{Art}$ will denote the category of Artinian local $k$-algebras with quotient field $k$.
By $\mathbf{Sets}$, we denote the category of sets which its morphisms are maps between sets.
Let $F$ and $G$ be functors from $\mathbf{Art}$ to $\mathbf{Sets}$. For two functors
$F,G: \mathbf{Art} \rightarrow \mathbf{Sets}$
the following is the notion of smoothness between morphisms of $F$ and $G$
which has been introduced in \cite{M. Sch.}:\\
\\
A morphism $D:F\rightarrow G$ between covariant functors $F$ and $G$ is said to be a smooth morphism of functors if for any surjective
morphism $\alpha:B\rightarrow A$, with $\alpha \in \Mor(\textbf{Art})$,
the morphism
$$F(B)\rightarrow F(A)\underset{G(A)}{\times}G(B)$$
is a surjective map in $\mathbf{Sets}$.\\
Note that this notion of smoothness
is a notion for morphisms between special functors, i.e. functors from the category $\mathbf{Art}$
to the category $\mathbf{Sets}$, while the concepts
for smoothness which we introduce in this paper are notions for functors, but not for morphisms
between them. \\
\\
A functor $F:\textbf{Art}\rightarrow \mathbf{Sets}$ is said to be a
deformation functor if it satisfies in definition 2.1. of \cite{M. Man.}. For a fixed field $k$ the schemes in this paper are schemes over the
scheme $\Spec(k)$ otherwise it will be stated.

\section{First Smoothness notion and some examples}

{\bf 1.1 Definition:}
Let $M$ and $C$ be two categories. We say that the category $C$ is a
 multicategory over $M$ if there exists a functor
$T:C\rightarrow M$, in which
for any
object $A$ of $M$, $T ^{-1}(A)$ is a full subcategory of $C$. \\
Let
$C$
and $\overline{C}$ be two multicategories over $M$ and $\overline{M}$ respectively. A morphism of
multicategories $C$ and $\overline{C}$ is a couple $(u,\nu)$ of functors,
 with $u:C \rightarrow \overline{C}$ and $\nu:M\rightarrow \overline{M}$ such that the following diagram is commutative:\\
$$\begin{array}{ccccc}
C &\overset{f}\rightarrow&M  \\
u \downarrow& & \downarrow \nu\\
\overline{C}& \rightarrow & \overline{M}\\
\end{array}$$ \\
The category of modules
 over the category of
rings and the category of sheaves of modules
 over the category of schemes are examples of
multicategories.\\

\noindent{\bf 1.2 Definition:} For a $S$-scheme $X$ and $A\in \mathbf{Art}$, we say that $\mathcal{X}$ is a
$S$-deformation of $X$ over $A$ if there is a commutative
diagram:
$$\begin{array}{ccccc}
X & \rightarrow & \mathcal{X}\\
\downarrow & & \downarrow \\
S & \rightarrow & S\underset{k}{\times}A \\
\end{array}$$
in which $X$ is a closed subscheme of $\mathcal{X}$, the scheme $\mathcal{X}$ is flat over
$S\underset{k}{\times}A$ and one has $X \cong S\underset{S\underset{k}{\times}A}{\times}\mathcal{X}$. \\
Note that in the case $S=\Spec(k)$, we would
 have the usual deformation notion and as in
the usual case the set of
 isomorphism classes of first order
 $S$-deformations of $X$ is a $k$-vector
space. The addition of two
 deformations $(\mathcal{X}_{1},\mathcal{O}_{\mathcal{X}_{1}})$ and $(\mathcal{X}_{2},\mathcal{O}_{\mathcal{X}_{2}})$ is denoted by
$(\mathcal{X}_{1}\underset{X}{\bigcup}\mathcal{X}_{2},\mathcal{O}_{\mathcal{X}_{1}}\underset{\mathcal{O}_{X}}{\times}\mathcal{O}_{\mathcal{X}_{2} })$.\\

\noindent{\bf 1.3 Definition:}
{\bf i)} Let $C$ be a category. We say $C$ is a category with enough deformations, if for any object $c$ of $C$,
one can associate a deformation functor. We will denote the associated deformation functor of $c$, by $D_{c}$.
Moreover for any $c\in \Obj(C)$ let
$D_{c}(k[\epsilon])$ be the tangent space of $c$, where $k[\epsilon]$ is the ring of dual numbers.\\
{\bf ii)} Let $C_{1}$ and $C_{2}$ be two
multicategories with enough deformations
 over $\Sch/k$, and
$(F,id)$ be a morphism between
them.
 We say $F$ is a smooth functor if it has the following
properties:\\
{\bf 1 :} For any object $M$ of $C_{1}$, if $M_{1}$ is a deformation of $M$ in
$C_{1}$ then $F(M_{1})$ is a deformation of $F(M)$ on $A$ in $C_{2}$.\\
{\bf 2 :} The map
$$
\begin{array}{ccc}
D_{M}(k[\varepsilon])&\rightarrow&D_{F(M)}(k[\varepsilon])\\
\mathcal{X}\!\!\!\!\!\!\!\!\!\!&\mapsto&\!\!\!\!\!\!\!\!\!\!F(\mathcal{X})
\end{array}
$$
is a morphism of tangent spaces.\\

\noindent The following are examples of categories with enough deformations:\\
1) Category of schemes over a field $k$.\\
2) Category of coherent sheaves on a scheme $X$.\\
3) Category of line bundles over a scheme.\\
4) Category of algebras over a field $k$.\\

\noindent We will need the following lemma to present an example of smooth functors:
\begin{lem}\label{lem1.1}
Let $X$, $X_{1}$, $X_{2}$ and $\mathcal{X}$ be schemes over a fixed scheme $S$.
Assume that the following diagram of morphisms between schemes is a commutative diagram.\\
\vspace{-.5cm}
\begin{center}
\unitlength .7500mm 
\linethickness{0.4pt}
\ifx\plotpoint\undefined\newsavebox{\plotpoint}\fi 
\begin{picture}(20,30)(30,90)
\put(25,115){\makebox(0,0)[cc]{$X$}}
\put(60,115){\makebox(0,0)[cc]{$X_1$}}
\put(25,85){\makebox(0,0)[cc]{$X_2$}}
\put(60,85){\makebox(0,0)[cc]{$\mathcal{X}$}}
\put(52.25,115.5){\vector(1,0){.07}}
\put(31,115.5){\line(1,0){21.25}}
\put(53.25,85.25){\vector(1,0){.07}}
\put(30.25,85.25){\line(1,0){23}}
\put(25,92.25){\vector(0,-1){.07}}
\put(25,110.25){\line(0,-1){18}}
\put(60,92.75){\vector(0,-1){.07}}
\put(60,110.5){\line(0,-1){17.75}}
\put(40,120){\makebox(0,0)[cc]{$i_1$}}
\put(40,80){\makebox(0,0)[cc]{$i_2$}}
\put(66,100){\makebox(0,0)[cc]{$g$}}
\end{picture}
\end{center}
\ \\
If $i_{1}$ is homeomorphic on its image, then so is $i_2$.
\end{lem}
\begin{proof}{
See Lemma $(2.5)$ of \cite{K. Sch.}.
}
\end{proof}
\begin{example}\label{exam0}
Let $Y$ be a flat scheme over
$S$. Then the fibered product by $Y$ over $S$ is smooth.
 More precisely, the functor:
$$\begin{array}{ccc}
F:\Sch/S&\rightarrow&\Sch/Y\\
F(X)\!\!\!\!\!\!\!\!\!\!\!\!\!&=&X\underset{S}{\times}Y
\end{array}$$
is smooth.
\end{example}
Let $X$ be a closed subscheme of $\mathcal{X}$. Then $X\underset{S}{\times}Y$ is a
closed subscheme of $\mathcal{X}\underset{S}{\times}Y$.
To get the flatness of  $\mathcal{X}\underset{S}{\times}Y$  over $S\underset{k}{\times}A$, it suffices to
has flatness of $Y$ over $S$.
It can also be verified easily that the isomorphism:
$$(\mathcal{X}\underset{S}{\times}Y)\underset{S\underset{k}{\times}A}{\times}S\cong X\underset{S}{\times}Y$$
is valid.
Therefore $\mathcal{X}\underset{S}{\times}Y$ is a $S$-deformation of $X\underset{S}{\times}Y$ if $\mathcal{X}$ is such a
deformation of $X$. This verifies
the first condition of item $(\mathbf{ii})$ of definition 1.3.
 To prove the second condition we need the following:
\begin{lem} \label{lem1.2}
Let $Y$, $X_{1}$ and $X_{2}$ be $S$-schemes. Assume that $X$ is a closed subscheme of $X_{1}$ and $X_{2}$.
 Then we have the following isomorphism:
\begin{center}
$(X_{1}\underset{X}{\bigcup} X_{2})\underset{S}{\times}Y\cong
(X_{1}\underset{S}{\times}Y)\underset{X\underset{S}{\times}Y}{\bigcup}(X_{2}\underset{S}{\times}Y)$.
\end{center}
\end{lem}
\begin{proof}{
For simplicity we set:
$$X_{1}\underset{X}{\cup}X_{2}=\mathcal{X} \qquad , \qquad
(X_{1}\underset{S}{\times}Y)\underset{X\underset{S}{\times}Y}{\bigcup}(X_{2}\underset{S}{\times}Y)=\mathcal{Z}$$
By universal property of $\mathcal{Z}$ we have a morphism
$\theta:\mathcal{Z}\rightarrow\mathcal{X}\underset{S}{\times}Y$.
 We prove that $\theta$ is an isomorphism.
Let $i_{1}:X_{1}\rightarrow \mathcal{X}$, $i_{2}: X_{2}\rightarrow \mathcal{X}$,
$j_{1}:X_{1}\underset{S}{\times}Y\rightarrow \mathcal{Z}$ and
$j_{2}:X_{2}\underset{S}{\times}Y\rightarrow \mathcal{Z}$
be the
inclusion morphisms.
Set theoretically we have:
$$\begin{array}{cccc}
j_{1}(X_{1}\underset{S}{\times}Y)\bigcup j_{2}(X_{2}\underset{S}{\times}Y)&=&\mathcal{Z}& \qquad(\mathbf{\I})\\
i_{1}(X_{1})\bigcup i_{2}(X_{2})&=&\mathcal{X} & \qquad(\II)
\end{array}$$
Now consider the following commutative diagrams:
\vspace{-.2cm}
\begin{center}
\unitlength 0.50mm 
\linethickness{0.4pt}
\ifx\plotpoint\undefined\newsavebox{\plotpoint}\fi 
\begin{picture}(70,100)(0,40)
\put(10,100){\makebox(0,0)[cc]{$X$}}
\put(40,130){\makebox(0,0)[cc]{$X_1$}}
\put(70,100){\makebox(0,0)[cc]{$\mathcal{X}$}}
\put(40,70){\makebox(0,0)[cc]{$X_2$}}
\put(34.75,127.25){\vector(1,1){.14}}
\multiput(13.75,103.5)(.067307692,.076121795){312}{\line(0,1){.076121795}}
\put(67.25,105.75){\vector(1,-1){.14}}
\multiput(45.5,127.5)(.067337461,-.067337461){323}{\line(0,-1){.067337461}}
\put(35.75,74.75){\vector(1,-1){.14}}
\multiput(13.5,97)(.067424242,-.067424242){330}{\line(0,-1){.067424242}}
\put(66.25,97.5){\vector(1,1){.14}}
\multiput(44.5,74)(.067337461,.072755418){323}{\line(0,1){.072755418}}
\put(20,120){\makebox(0,0)[cc]{$f$}}
\put(58,120){\makebox(0,0)[cc]{$i_1$}}
\put(58,80){\makebox(0,0)[cc]{$i_2$}}
\put(20,80){\makebox(0,0)[cc]{$g$}}
\end{picture}\\
\vspace{-.4cm}
\unitlength .500mm 
\linethickness{0.4pt}
\ifx\plotpoint\undefined\newsavebox{\plotpoint}\fi 
\begin{picture}(138.5,100)(0,20)
\put(23,77){\makebox(0,0)[cc]{$X\underset{S}{\times}Y$}}
\put(57,113){\makebox(0,0)[cc]{$X_1\underset{S}{\times} Y$}}
\put(57,42){\makebox(0,0)[cc]{$X_2\underset{S}{\times}Y$}}
\put(130,115){\makebox(0,0)[cc]{$\mathcal{Z}$}}
\put(130,42){\makebox(0,0)[cc]{$\mathcal{X}\underset{S}{\times}Y$}}
\put(51.5,109){\vector(1,1){.07}}
\multiput(27.25,84.75)(.0337078652,.0337078652){700}{\line(0,1){.0337078652}}
\put(51.5,49.3){\vector(1,-1){.07}}
\multiput(26.5,74.5)(.0337273992,-.0337273992){700}{\line(0,-1){.0337273992}}
\put(120,115.5){\vector(1,0){.07}}
\put(75,115.5){\line(1,0){45}}
\put(118,45.25){\vector(1,0){.07}}
\put(75,45.25){\line(1,0){40}}
\put(129.5,51.25){\vector(0,-1){.07}}
\multiput(129.75,110.5)(-.03125,-7.40625){8}{\line(0,-1){7.40625}}
\put(121.5,52.5){\vector(1,-1){.07}}
\multiput(62.5,111)(.0340253749,-.0337370242){1734}{\line(1,0){.0340253749}}
\put(122.25,110){\vector(1,1){.07}}
\multiput(62.5,50)(.0337380011,.0338791643){1771}{\line(0,1){.0338791643}}
\put(35,100){\makebox(0,0)[cc]{$g_1$}}
\put(90,121){\makebox(0,0)[cc]{$j_1$}}
\put(35,57.75){\makebox(0,0)[cc]{$g_2$}}
\put(90,36){\makebox(0,0)[cc]{$h$}}
\put(138,80){\makebox(0,0)[cc]{$\theta$}}
\put(108,70){\makebox(0,0)[cc]{$e$}}
\put(75,70){\makebox(0,0)[cc]{$j_2$}}
\end{picture}
\end{center}
Let $z\in \mathcal{X}\underset{S}{\times}Y$, $\alpha=P_{\mathcal{X}}(z)\in \mathcal{X}$ and
$\beta=P_{Y}(z)\in Y$ in which $P_{\mathcal{X}}$ and $P_{Y}$ are the first and second projections from $\mathcal{X} \underset{S}{\times}Y$
to $\mathcal{X}$ and $Y$ respectively.
 Then by relation $(\II)$ one has $ \alpha\in i_{1}(X_{1})$ or $ \alpha\in i_{2}(X_{2})$. If
$\alpha=i_{1}(\alpha_{1})\in i_{1}(X_{1})$, then
$\alpha_{1}$ and $\beta$
go to the same element in S by $\eta_{X_{1}}$ and $\eta_{Y}$ in which $\eta_{X_{1}}:X_{1}\rightarrow S$ and
$\eta_{Y}:Y\rightarrow S$ are the maps which make $X_{1}$ and $Y$ schemes over $S$.
Therefore there exists an element
$\gamma$ in $X_{1}\underset{S}{\times}Y$
such that $\overline{P}_{X_{1}}(\gamma)=\alpha_{1}$ and
$\overline{P}_{Y}(\gamma)=\beta$
in which $\overline{P}_{X_{1}}$ and $\overline{P}_{Y}$ are
the first and second projections from  $X\underset{S}{\times}Y$ to $X_{1}$ and $Y$ respectively.
By universal property of fibered products $\gamma$ belongs to $\mathcal{X}\underset{S}{\times}Y$
 and $\theta(\gamma)=z$. The proof for the case $\alpha \in i_{2}(X)$ is similar.
This implies that $\theta$ is surjective.\\
For injectivity of $\theta$ assume that
$\theta(z_{1})=\theta(z_{2})$.
 The relation $(\I)$ implies that $z_{1}$ and $z_{2}$ belong to $\im(j_{1})\bigcup \im(j_{2})$. Set
$z_{1}=j_{1}(c_{1})$ and $z_{2}=j_{2}(c_{2})$.
There are two cases: if $z_{1}, z_{2} \in \im(j_{1})\cap \im(j_{2})$,
then the lemma \ref{lem1.1} implies
 $e(c_{1})\neq e(c_{2})$ when $c_{1}\neq c_{2}$.
Now by commutativity of the subdiagram:
\begin{center}
\unitlength .7500mm 
\linethickness{0.4pt}
\ifx\plotpoint\undefined\newsavebox{\plotpoint}\fi 
\begin{picture}(30,30)(30,90)
\put(20,115){\makebox(0,0)[cc]{$X_1\underset{S}{\times}Y$}}
\put(70,114){\makebox(0,0)[cc]{$\mathcal{X} \underset{S}{\times}Y$}}
\put(70,82){\makebox(0,0)[cc]{$\mathcal{Z}$}}
\put(59,115.5){\vector(1,0){.07}}
\put(32.75,115.5){\line(1,0){26.25}}
\put(70.25,108){\vector(0,1){.07}}
\put(70.25,86.5){\line(0,1){21.5}}
\put(65.75,86.25){\vector(3,-2){.07}}
\multiput(28,110.25)(.0530196629,-.0337078652){712}{\line(1,0){.0530196629}}
\put(39.25,95.25){\makebox(0,0)[cc]{$j_1$}}
\put(75.25,98.25){\makebox(0,0)[cc]{$\theta$}}
\end{picture}
\end{center}
we have $\theta(z_{1})\neq \theta(z_{2})$ when $z_{1}\neq z_{2}$.\\
Otherwise assume that $z_{1}\in \im(j_{1})$ and $z_{2}\in \im(j_{2})- \im(j_{1})$. In this case
one can see easily that
$i_{1}\overline{P}_{X_{1}}(c_{1})=i_{2}q_{2}(c_{2})$ in which $q_{2}$ is
the first projection from
$X_{2}\underset{S}{\times}Y$ to $X_{2}$.
Since $\mathcal{X}$ is the fibered sum of $X_{1}$ and $X_{2}$,
there exists an
element $x\in X$ such that $i_{1}f(x)=i_{2}g(x)$, $f(x)=\overline{P}_{X_{1}}(c_{1})$
and $g(x)=q_{2}(c_{2})$.\\
Set $y=p_{2}e(c_{1})$ in which $p_{2}$ is the second projection from
$\mathcal{X}\underset{S}{\times}Y$ to $Y$.
By a diagram chasing we see that
$x$ and $y$ go to
the same element in S. This implies that there exists an element $\epsilon$ in $X\underset{S}{\times}Y$ which is mapped to $x$
 and $y$ by first and second projections, respectively. Also it is easy to see that
the equalities $g_{1}(x,y)=c_{1}$ and $g_{2}(x,y)=c_{2}$ are valid. Since $\mathcal{Z}$ is the fibered
sum of $X_{1}\underset{S}{\times}Y$ and $X_{2}\underset{S}{\times}Y$ on $X\underset{S}{\times}Y$, we have
$z_{1}=z_{2}$ which means that $\theta$ is injective. This together with the surjectivity of $\theta$ implies that $\theta$
is bijective.
Continuity of $\theta$ and its inverse, follow by a diagram chasing.\\
Finally we should prove that $\mathcal{O}_{\mathcal{X}\underset{S}{\times}Y}\cong \mathcal{O}_{Z}$. Since the claim is
local, it is sufficient to prove it for affine schemes. Let $\mathcal{X}$ be an affine scheme, so $X_{1}$,
$X_{2}$ and $X$ are affine schemes, since they are closed
subschemes of $\mathcal{X}$ each one defined by a nilpotent
 sheaf of ideals.
Set $\mathcal{X}=\Spec(A)$, $X_{1}=\Spec(A_{1})$, $X_{2}=\Spec(A_{2})$, $X=\Spec(A_{0})$,
$Y=\Spec(B)$ and $S=\Spec(C)$. The isomorphism $\mathcal{O}_{\mathcal{X}\underset{S}{\times}Y}\cong \mathcal{O}_{Z}$ reduces to the following
isomorphism:
$$(A_{1}\underset{A_{0}}{\times}A_{2})\underset{C}{\otimes}B\cong
(A_{1}\underset{C}{\otimes}B)\underset{A_{0}\underset{C}{\otimes}B}{\times}(A_{2}\underset{C}{\otimes}B).$$
Define a morphism as follows:
$$\begin{array}{ccc}
d:(A_{1}\underset{A_{0}}{\times}A_{2})\underset{C}{\otimes}B&\rightarrow&
(A_{1}\underset{C}{\otimes}B)\underset{A_{0}\underset{C}{\otimes}B}{\times}(A_{2}\underset{C}{\otimes}B)\\
d((a_{1},a_{2})\otimes b)\!\!\!\!\!\!\!\!\!\!\!\!\!\!\!\!\!\!\!\!&=&\!\!\!\!\!\!\!\!\!\!\!\!\!\!\!\!\!\!\!\!\!\!\!\!\!\!\!\!\!\!
\!\!(a_{1}\otimes b,a_{2}\otimes b).
\end{array}$$
By a simple commutative algebra argument it can be shown that this is in fact an
isomorphism. This completes the proof of lemma.}
\end{proof}
This lemma shows that the fibered product functor, induces an
additive homomorphism on tangent spaces. To check linearity
with respect to scalar multiplication, take an element $a$ in the field $k$.
 Multiplication by $a$ is a ring homomorphism on $D$. This homomorphism induces a
morphism from $S\underset{k}{\times}D$ to $S\underset{k}{\times}D$ and scalar multiplication on
$t_{D_{X}}$, comes from composition of this map with
$\pi$. In other words this gives a map from $\mathcal{X}\underset{S}{\times}Y$
into $\mathcal{X}\underset{S}{\times}Y$. These together give the
linearity of homomorphism induced from $F$ with respect to scalar
multiplication.\\
This observation together with the lemma \ref{lem1.2}, give the smoothness of the fibered product functor.
\begin{lem}\label{lem1.3}
Let $X$ and $Y$ be arbitrary schemes and assume that
there exist morphisms
$h$ and $g$ from $\eta$ to $\eta_{1}$ and $\eta_{2}$, where
$\eta$, $\eta_{1}$, $\eta_{2}$ are sheaves of
$\mathcal{O}_{X}$-modules on the scheme $X$.
Then for any morphism $f:X\rightarrow Y$ we have the following isomorphisms:
$$
\begin{array}{ccc}
f_{*}(\eta_{1}\underset{\eta}{\times}\eta_{2})\!\!\!\!\!&\cong &\!\!\!\!\!f_{*}(\eta_{1}) \underset{f_{*}(\eta)}{\times}
f_{*}(\eta_{2}) \\
f^{*}(\rho_{1}\underset{\rho}{\times}\rho_{2})\!\!\!\!\!&\cong&\!\!\!\!\!
f^{*}(\rho_{1})\underset{f^*(\rho)}{\times}
f^{*}(\rho_{2}).
\end{array}
$$
\end{lem}
\begin{proof}{
 For the first isomorphism, it is enough to consider
the definition of direct image of sheaves.\\
To prove the second one, assume that $(M_{i})_{i\in I}, (N_{i})_{i\in I}$ and $(P_{i})_{i \in I}$
are direct systems of modules over a directed set $I$. We have to prove that
$$\lim_{i\in I}(M_{i}\underset{P_{i}}{\times}N_{i})\cong (\lim_{i\in
I}(M_{i}))\underset{(\lim_{i\in I}(P_{i}))}{\times}(\lim_{i\in I}(N_{i})).$$
The above isomorphism can be proved by elementary calculations and using elementary properties of
direct limits.}
\end{proof}
\begin{example}\label{exam1}
Let $f:X\rightarrow Y$ be a flat morphism of schemes.
Then $f_{*}$ and $f^{*}$ are smooth functors.
\end{example}
In fact
let $\eta$ be a coherent sheaf on $X$ and $\eta_{1}\in \Coh(X\underset{k}{\times}D)$ be a
deformation of $\eta$. By these assumptions we would have:
$$(f_{*}(\eta))\underset{D}{\otimes}k=f_{*}(\eta_{1}\underset{D}{\otimes}k)=f_{*}(\eta).$$
Moreover
$f_{*}(\eta_{1})$ is flat on $D$, because $\eta$ is flat on $D$. This implies that $f_{*}$
satisfies in the first condition of smoothness. The second one is the first
isomorphism of lemma \ref{lem1.3}. Therefore $f_{*}$ is smooth.
Smoothness of $f^{*}$ is similar to that of $f_{*}$.

Assuming this notion of smoothness we can generalize another aspect of geometry
to categories.

\noindent{\bf 1.9 Definition:} Let $C$ be a category with enough deformations.
We define the tangent category of $C$, denoted by $TC$, as follows:
$$
\begin{array}{ccc}
\Obj(TC)\!\!\!\!\!\!\!\!\!\!&:=&\underset{c\in\Obj(C)}{\bigcup} T_{c}C\\
\Mor_{TC}(\upsilon,\omega)&:=&\!\!\!\!\!\!\!\!\!\!\Mor(V,W)
\end{array}
$$
which by $T_{c}C$, we
mean the tangent space
of $D_{c}$.
Moreover $\upsilon$ and $\omega$ are
first order deformations of $V$ and $W$.
\begin{rem}\label{rem1}
(i) It is easy to see that a smooth functor induces a covariant functor on the tangent categories.\\
(ii) Let $C$ be an abelian category. Then its tangent category is also abelian.
\end{rem}
The following is a well known suggestion of A. Grothendieck:
Instead of working with a space,
 it is enough to
work on the category of quasi coherent sheaves on this space.
This suggestion
was formalized and proved by P. Gabriel for noetherian schemes and in its general form by A. Rosenberg.
To do this,
Rosenberg associates a locally ringed space to an abelian category $A$. In a
special case he gets the following:
\begin{thm}\label{Th1.1}
 Let $(X,\mathcal{O}_{X})$ be a locally ringed space
 and let $A=\QCoh(X)$. Then
 $$(\Spec(A),\mathcal{O}_{ \Spec(A)})=(X,\mathcal{O}_{X})$$
where
$\Spec(A)$ is the ringed space which is constructed from an abelian category by A. Rosenberg.
\end{thm}
\begin{proof}{
See Theorem $(A.2)$ of \cite{A. L. R}.
}
\end{proof}
The definition of tangent category and theorem 4 motivates the following questions which the authors
could not find any positive or negative answer to them until yet.\\

\noindent{\bf Question 1:}
For a fixed scheme $X$ consider $T\QCoh(X)$ and $TX$, the tangent category of
category of quasi coherent sheaves on $X$ and the tangent bundle of $X$ respectively. Can $TX$ be recovered from $T\QCoh(X)$
by Rosenberg construction?\\
\textbf{Question 2:} Let $\mathcal{M}$
 be a moduli family with moduli space $M$.
Consider $\mathcal{M}$ as a category and consider its tangent category $T\mathcal{M}$. Is there a
 reconstruction from $T\mathcal{M}$ to $TM$?

\section{Second Smoothness Notion}
\textbf{Definition 3.1 :}
 Let $F:\Sch/k\rightarrow \Sch/k$ be a functor with the following property:\\
 For any scheme X and an algebra $A\in \Obj(\textbf{Art})$, $F(\mathcal{X})$ is a deformation of $F(X)$ over $A$
 if $\mathcal{X}$ is a deformation of $X$ over $A$.\\
We say $F$ is
 smooth at $X$, if the morphism of
 functors
$$\Theta_{X}:D_{X}\rightarrow D_{F(X)}$$
is a smooth
morphism of functors in the sense of Schlessinger (See \cite{M. Sch.}). $F$ is said to be smooth if for
any object $X$ of $\Sch/k$, the morphism of functors   $\Theta_{X}$ is
smooth.\\

\noindent The following lemma describes more properties of smooth functors.
\begin{lem}\label{lem2.1}
$(a)$ Assume that $C_{1}$, $C_{2}$ and $C_{3}$ are multicategories over the category $\Sch/k$.
Let $F_{1}:C_{1}\rightarrow C_{2}$ and $F_{2}:C_{2}\rightarrow C_{3}$ be
smooth functors with the first notion.
Then so is their composition.\\
$(b)$ Let $F_{1}:\Sch/k \rightarrow \Sch/k$ and $F_{2}:\Sch/k \rightarrow \Sch/k$ be
smooth functors with second notion. Then so is their composition.\\
$(c)$ Let $F:\Sch/k \rightarrow \Sch/k$ and $G:\Sch/k \rightarrow \Sch/k$ be functors to which $F$ and $GoF$ are smooth with second notion.
Then $G$ is a smooth functor.\\
$(d)$ Let $F,G,H: \Sch/k\rightarrow \Sch/k$ be smooth functors in the sense of second notion with
morphisms of functors $F\rightarrow G$ and $H\rightarrow G$
between them. Then the functor $F\underset{G}{\times}H$ is smooth functor with the second one.
\end{lem}
\begin{proof}{
Part $(a)$ of lemma is trivial.\\
$(b)$ Let $X\in \Sch/k$ and $B\rightarrow A$ be a surjective morphism  in $\mathbf{Art}$.
By smoothness of $F_{1}$, $F_{2}$ and by remark $2.4$ of \cite{M. Sch.}, there exists a surjective map
$$\Theta_{F_{2}(X),F_{2}oF_{1}(X)}: D_{F_{2}oF_{1}(X)}(B)\underset{D_{F_{2}oF_{1}(X)}(A)}{\times}D_{X}(A)\rightarrow D_{F_{1}(X)}(B)\underset{D_{F_{1}(X)}(A)}{\times}D_{X}(A)$$
such that we have
$$\Theta_{X,F_{2}oF_{1}(X)}=\Theta_{F_{2}(X),F_{2}oF_{1}(X)}o\Theta_{X,F_{2}(X)}$$
in which  $\Theta_{X,F_{2}(X)}$
is the surjective map induced by smoothness of $F_{2}$. From this equality it follows the map $\Theta_{X,F_{2}oF_{1}(X)}$ is surjective
immediately.\\
$(c)$ For a scheme $X$ in the category $\Sch/k$ consider a surjective morphism $B\rightarrow A$ in $\mathbf{Art}$. By smoothness of
$F$, the morphism $D_{X}\rightarrow D_{F(X)}$ is a surjective morphism of functors. Now apply  proposition $(2.5)$ of $\cite{M. Sch.}$ to finish the
proof.\\
$(d)$ Let $X\in \Sch/k$ and $B\rightarrow A$ be a surjective morphism  in $\mathbf{Art}$. Consider the following commutative diagram:
\vspace{-.5cm}
\begin{center}
\unitlength .7500mm 
\linethickness{0.4pt}
\ifx\plotpoint\undefined\newsavebox{\plotpoint}\fi 
\begin{picture}(30,30)(30,90)
\put(20,115){\makebox(0,0)[cc]{$D_{X}$}}
\put(70,115){\makebox(0,0)[cc]{$D_{F(X)}$}}
\put(70,80){\makebox(0,0)[cc]{$D_{G(X)}$}}
\put(59,115.5){\vector(1,0){.07}}
\put(32.75,115.5){\line(1,0){26.25}}
\put(70.25,108){\vector(0,1){.07}}
\put(70.25,86.5){\line(0,1){21.5}}
\put(61.75,86.25){\vector(3,-2){.07}}
\multiput(24,110.25)(.0530196629,-.0337078652){712}{\line(1,0){.0530196629}}
\put(39.25,95.25){\makebox(0,0)[cc]{$$}}
\put(75.25,98.25){\makebox(0,0)[cc]{$$}}
\end{picture}
\end{center}
\vspace{+.5cm}
Since the morphisms of functors
$D_{X}\rightarrow D_{F(X)}$ and
 $D_{X}\rightarrow D_{G(X)}$  are smooth morphisms of functors,
proposition $2.5(iii)$ of $\cite{M. Sch.}$ implies that
$D_{F(X)}\rightarrow D_{G(X)}$ is a smooth morphism of functors.
Similarly $D_{H(X)}\rightarrow D_{G(X)}$ is a smooth morphism of functors. Again by $2.5(iv)$ of $\cite{M. Sch.}$, the morphism of functors:
$$D_{H(X)}\underset{D_{G(X)}}{\times}D_{F(X)}\rightarrow D_{H(X)}$$
is a smooth morphism of functors. Since
in the diagram:
\begin{center}
\unitlength .7500mm 
\linethickness{0.4pt}
\ifx\plotpoint\undefined\newsavebox{\plotpoint}\fi 
\begin{picture}(30,30)(30,90)
\put(25,115){\makebox(0,0)[cc]{$D_{X}$}}
\put(81,113){\makebox(0,0)[cc]{$D_{H(X)}\underset{D_{G(X)}}{\times}D_{F(X)}$}}
\put(81,80){\makebox(0,0)[cc]{$D_{H(X)}$}}
\put(59,115.5){\vector(1,0){.07}}
\put(32.75,115.5){\line(1,0){26.25}}
\put(80.25,108){\vector(0,1){.07}}
\put(80.25,86.5){\line(0,1){21.5}}
\put(71.75,86.25){\vector(3,-2){.07}}
\multiput(34,110.25)(.0530196629,-.0337078652){712}{\line(1,0){.0530196629}}
\put(39.25,95.25){\makebox(0,0)[cc]{$$}}
\put(75.25,98.25){\makebox(0,0)[cc]{$$}}
\end{picture}
\end{center}
\vspace{+.5cm}
the morphisms
$D_{X}\rightarrow D_{H(X)}$ and
$D_{H(X)}\underset{D_{G(X)}}{\times}D_{F(X)}$
are smooth morphisms of functors, part $(c)$ of this lemma implies that $D_{H(X)}\underset{D_{G(X)}}{\times}D_{F(X)}$
is smooth. This completes the proof.
}
\end{proof}
\begin{rem}\label{rem22} $\textbf{(i)}$ The same proof works to generalize part $(c)$ of lemma \ref{lem2.1} as follows:\\
$(\acute{c})$ Let $F:\Sch/k \rightarrow \Sch/k$ and $G:\Sch/k \rightarrow \Sch/k$ be functors with $GoF$ smooth
and $F$ surjective in the level of deformations in the sense that for any $X\in \Sch/k$ and any $A\in \Obj(\mathbf{Art})$
the morphism $D_{X}(A)\rightarrow D_{F(X)}(A)$ is surjective in $\mathbf{Art}$. Then $G$ is smooth.\\
$\textbf{(ii)}$ One may ask to find a criterion to determine smoothness of a functor. We could not get a complete answer to this question.
But
by the following fact, one may answer the question at least partially:\\
A functor $F:\Sch/k \rightarrow \Sch/k$ is not smooth at $X$ if
there exists an algebra $A\in \mathbf{Art}$ such that the map $D_{X}(A)\rightarrow D_{F(X)}(A)$
is not surjective in $\mathbf{Art}$, (See \cite{M. Sch.}).
\end{rem}
\noindent Theorem \ref{Th2} relates the second smoothness notion to the hull of deformation functors.
Recall the hull of a functor is defined in \cite{M. Sch.}. We need the following:
\begin{lem}\label{lem2.2}
Let $F: \mathbf{Art} \rightarrow \Sets$ be a functor. Then its hulls are non-canonically isomorphic if there
exist.
\end{lem}
\begin{proof}{
See Proposition $2.9$ of \cite{M. Sch.}.
}
\end{proof}
\begin{thm}\label{Th2}
Let $F:\Sch/k\rightarrow \Sch/k$ be a
functor and for a scheme $X$ the functor $F$ has the following
properties:\\
$(a)$ $F(\mathcal{X})$ is a deformation of $F(X)$ if $\mathcal{X}$ is a deformation of $X$.\\
$(b)$ The functor $F$ induces isomorphism on tangent spaces.\\
Then
 $F$ is smooth at $X$ if and only if $(R,F(\xi))$ is a hull of $D_{F(X)}$ whenever $(R,\xi)$ is a hull of $D_{X}$.
\end{thm}
\begin{proof}{
To prove the Theorem it is enough to apply
$(b), (c)$ of lemma \ref{lem2.1}, and lemma \ref{lem2.2}
to the functors
$$\Theta_{X}:D_{X}\rightarrow D_{F(X)} \quad,\quad h_{R,X}:h_{R}\rightarrow D_{X} \quad,\quad h_{R,F(X)}:h_{R}\rightarrow  D_{F(X)}.$$
}
\end{proof}
\vspace{-.3cm}
For a scheme $X$ let:
\begin{center}
 \{pairs $(\mathcal{X},\Omega_{\mathcal{X}/k})$ which $\mathcal{X}$ is an infinitesimal deformation
of $X$ over $A$ \}
\end{center}
be the isomorphism classes of fibered deformations of $X$.\\
In the following example we use this notion of deformations of schemes.
\begin{example}\label{exam2}
The functor defined by:
$$\begin{array}{ccc}
F:\Sch/k &\rightarrow& \QCoh \\
F(X)\!\!\!\!\!\!\!\!\!\!&=&\Omega_{X/k}
\end{array}$$
is a smooth functor.
\end{example}
Note that if one considers deformations of $\Omega_{X/k}$ as usual case, the
above functor will not be smooth. The usual deformation of $ \Omega_{X/k}$ can
be described as simultaneous deformation of an object,
and differential forms on that object.
Also this observation is valid for
$TX$ and $\omega_{X}$ instead of $\Omega_{X}$.
\begin{rem}\label{rem2}
The first and second smoothness notions are in general different.
Note that a functor which is smooth with the second
notion induces surjective maps on tangent spaces.
Since the morphism induced on tangent spaces with first notion
of smoothness is not necessarily surjective,
a functor which is smooth in the sense of first notion
is not necessarily smooth with the sense of second notion. Also a functor which is smooth in the sense of second
notion can not be necessarily smooth with the first notion in general. In fact the map induced on tangent spaces
by second notion is not necessarily a linear map. It is easy to see that the example \ref{exam2} is smooth with both of the notions, but
examples \ref{exam0} and \ref{exam1} are smooth just in the sense of first one.
\end{rem}
\subsection{A Geometric interpretation}
Let $F$ be a smooth functor at $X$. By theorem \ref{Th2}, $X$ and
$F(X)$ have the same universal rings and this can be interpreted
as we are deforming $X$ and $F(X)$ simultaneously. Therefore we
have an algebraic language for simultaneous deformations. The
example \ref{exam2} can be interpreted as follows: we are deforming a
geometric space and an ingredient of that space, e.g. the
structure sheaf of the space or its sheaf of relative differential forms, and these
operations are smooth.

\subsection{Relation with smoothness of a morphism}

Let $\mathcal{M}$ be a moduli family of algebro - geometric objects with
a variety $M$ as its fine moduli space and suppose
$Y(m)\rightarrow M $ is the fiber on $m\in M$. With this
assumptions we would have the following bijections:
$$
\begin{array}{ccc}
T_{m,M}&\cong&\!\!\!\!\!\!\!\!\!\!\!\!\!\!\!\!\!\!\!\!\!\!\!\!\!\!\!\!\!\!\!\!\!\!\!\!\!\!
\!\!\!\!\!\!\!\!\!\!\!\!\!\!\!\!\!\!\!\!\!\!\!\!\!\!\!\!\!\!
\mbox{Hom}(\Spec(k[\epsilon]),M)\\
&\cong&\{\mbox{classes of first order deformations of X over A} \}
\end{array}
$$
In fact these bijections states that why deformations are
important in geometric usages. Now
suppose we have two moduli families $\mathcal{M}_{1}$ and $\mathcal{M}_{2}$
with
 varieties $M_{1}$ and $M_{2}$ as their fine moduli spaces. Also
 describe $\mathcal{M}_{1}$ and $\mathcal{M}_{2}$ as categories in which there
exists
 a smooth functor  $F$ between them.
 In this setting, if we have a morphism between
 them, induced from  $F$, then it is a smooth morphism.

\section{Third Smoothness Notion}

This notion of smoothness is completely motivated from Rosenberg's
reconstruction theorem, Theorem $(A.2)$ of \cite{A. L. R}. For this notion of smoothness we do not use deformation theory.

\noindent\textbf{3.1 Definition:} Let $F:C_{1} \rightarrow C_{2}$ be a functor between
abelian categories such that there exists a morphism
$$f:\Spec(C_{1})\rightarrow \Spec(C_{2})$$
induced by the functor $F$. We say $F$ is a smooth functor if $f$ is a smooth morphism of
schemes.
\begin{rem}\label{rem3}
$(a)$ Since this smoothness notion uses a language completely different from the two previous ones,
it does not imply non of them and vice versa. We did not verified this claim with details but it is not so
legitimate to expect that this smoothness implies the previous ones, because deformation theory is not consistent with the
Rosenberg construction. This observation together with the remark \ref{rem2} show that these three notions are
independent of each other, having nice geometric and algebraic meaning in their own rights separately.\\
$(b)$ It seems that a functor of abelian categories induces a morphism of schemes
in rarely cases. But the cases in which this happens are the
cases of enough importance to consider them. Here we mention some cases which
this happens.\\
\textbf{(i)} Let $f:X \rightarrow \Spec(k)$ be a morphism of finite type
between schemes. Then it can be shown $f$ is induced by
$$f_{*}:\QCoh(X) \rightarrow
\QCoh(\Spec(k))$$
by Rosenberg's construction. This example is important because it can be a source of motivation, to translate
notions from commutative case to noncommutative one.\\
\textbf{(ii)} Also the following result of Rosenberg is worth to note:\\
\begin{pro}
Let $A$ be an abelian category.\\
(a) For any topologizing subcategory $T$ of $A$, the
inclusion functor $T\rightarrow A$ induces an embedding $\Spec(T) \rightarrow
\Spec(A)$.\\
(b) For any exact localization $Q:A \rightarrow A/S $ and for any $P \in
\Spec(A)$, either $P \in \Obj(S)$ or $Q(P)\in \Spec(A/S)$; hence $Q$ induces an
injective map from $\Spec(A)-\Spec(S)$ to $\Spec(A/S)$.
\end{pro}
\begin{proof}{
See Proposition $(A.0.3)$ of \cite{A. L. R}.
}
\end{proof}
\end{rem}
\textbf{Acknowledgements:} The authors are grateful for referee/s carefully reading of the paper,
notable remarks and  valuable suggestions about
it.

\end{document}